\documentclass[12pt, twoside, leqno]{article}

\usepackage{amsfonts,amsmath,amsthm}
\usepackage{amssymb}
\usepackage{hyperref}
\usepackage{relsize}
\usepackage{graphicx}
\usepackage{enumitem}

\usepackage{graphicx}

\usepackage[T1]{fontenc}

\newtheorem{theorem}{Theorem}[section]

\newtheorem{lemma}[theorem]{Lemma}

\theoremstyle{definition}

\newtheorem*{xrem}{Remark}

\numberwithin{equation}{section}

\begin{document}
	
	
	\baselineskip=17pt

\title{On the average behavior of the Fourier coefficients of $j^{th}$ symmetric power $L$-function over a certain sequences of positive integers}

\author{Anubhav Sharma\footnote{Corresponding author}\\
	School of Mathematics and Statistics\\ 
	University of Hyderabad \\
	P.O, Prof. C.R.Rao Road, Gachibowli\\
	Hyderabad, India\\
	E-mail: 19mmpp02@uohyd.ac.in
	\and 
	Ayyadurai Sankaranarayanan\\
	School of Mathematics and Statistics\\ 
	University of Hyderabad \\
	P.O, Prof. C.R.Rao Road, Gachibowli\\
	Hyderabad, India\\
	E-mail: sank@uohyd.ac.in}

\date{}

\maketitle


\renewcommand{\thefootnote}{}



\renewcommand{\thefootnote}{\arabic{footnote}}
\setcounter{footnote}{0}

	
\begin{abstract}
	In this paper, we investigate the average behavior of the $n^{th}$ normalized Fourier coefficients of the $j^{th}$ ($j \geq 2$ be any fixed interger) symmetric power $L$-function (i.e., $L(s,sym^{j}f)$), attached to a primitive holomorphic cusp form $f$ of weight $k$ for the full modular group $SL(2,\mathbb{Z})$ over a certain sequences of positive integers. Precisely, we prove an asymptotic formula with an error term for the sum
	$$\mathlarger{\mathlarger{\sum}}_{\stackrel{a_{1}^{2}+a_{2}^{2}+a_{3}^{2}+a_{4}^{2}+a_{5}^{2}+a_{6}^{2}\leq {x}}{(a_{1},a_{2},a_{3},a_{4},a_{5},a_{6})\in\mathbb{Z}^{6}}}\lambda^{2}_{sym^{j}f}(a_{1}^{2}+a_{2}^{2}+a_{3}^{2}+a_{4}^{2}+a_{5}^{2}+a_{6}^{2}),$$ where $x$ is sufficiently large, and
	$$L(s,sym^{j}f):=\mathlarger{\mathlarger{\sum}}_{n=1}^{\infty}\dfrac{\lambda_{sym^{j}f}(n)}{n^{s}}.$$
	When $j=2$, the error term which we obtain, improves the earlier known result.
\end{abstract}

\section{Introduction}

Let $L(s,f)$ be the $L$-function associated with the primitive holomorphic cusp form $f$ of weight $k$ for the group $SL(2,\mathbb{Z})$ and let $\lambda_f(n)$ be the normalized $n^{th}$ Fourier coefficient of the Fourier expansion of $f(z)$ at the cusp $\infty$, i.e.,
$$f(z)=\mathlarger{\mathlarger{\sum}}_{n=1}^{\infty}\lambda_f(n)n^{\frac{k-1}{2}}e^{2\pi inz},$$ where $\Im(z)>0$. Then the $L$-function attached to $\lambda_f(n)$ is defined as 
$$L(s,f)=\mathlarger{\mathlarger{\sum}}_{n=1}^{\infty}\frac{\lambda_f(n)}{n^{s}}    $$ for $\Re(s)>1$, where $\lambda_f(n)$ are the eigenvalues of all the Hecke operators $T_n$. \\

Let $\chi$ be the Dirichlet character modulo $N$. Then $f$ is known as a modular form of weight $k$ and level $N$ with Nebentypus $\chi$, if
$$f \left( \frac{az+b}{cz+d} \right)=\chi(d)(cz+d)^{k}f(z)  $$ 
for all $z\in\mathbb{H}$ (upper half plane) and $\begin{pmatrix}
	a & b \\
	c & d 
\end{pmatrix} \in \Gamma_0(N)$, where $\Gamma_0(N)$ is the congruence subgroup, i.e., $$\Gamma_0(N)=\left\{
\begin{pmatrix}
	a & b \\
	c & d 
\end{pmatrix} \in SL(2,\mathbb{Z}): c\equiv 0\pmod N
\right\}.$$ \\

In 1974, Deligne \cite{deligne1974conjecture} proved that for any prime $p$ there exist complex numbers $\alpha(p)$ and $\beta(p)$ such that
\begin{eqnarray}
	\alpha(p)+\beta(p)=\lambda_f(p),
\end{eqnarray} 
and
\begin{eqnarray}
	|\alpha(p)|=|\beta(p)|=1=\alpha(p)\beta(p).
\end{eqnarray}
Then $L(s,f)$ can be written as 

\begin{equation*}
	L(s,f)=\prod_{p}^{} \left(1-\frac{\alpha(p)}{p^{s}}\right)^{-1}\left(1-\frac{\beta(p)}{p^{s}}\right)^{-1}.
\end{equation*}
Also, $|\lambda_f(n)|\leq d(n)$, where $d(n)$ is the divisor function.\\

The symmetric square $L$-function is defined as 
\begin{align*}
	L(s,sym^{2}f) & :=\mathlarger{\mathlarger{\sum}}_{n=1}^{\infty}\dfrac{\lambda_{sym^{2}f}(n)}{n^{s}} \\
	& =\prod_{p}^{} \left(1-\frac{\alpha^{2}(p)}{p^{s}}\right)^{-1}\left(1-\frac{\beta^{2}(p)}{p^{s}}\right)^{-1}\left(1-\frac{1}{p^{s}}\right)^{-1},  
\end{align*}
for $\Re(s)>1$, where $\lambda_{sym^{2}f}(n)$ is multiplicative.\\

Several authors have studied the average behavior of these Fourier coefficients. In 2006, Fomenko \cite{fomenko2006identities} was able to prove some results for the symmetric square $L$-functions. He showed that
$$\mathlarger{\mathlarger{\sum}}_{n\leq{x}}\lambda_{sym^{2}f}(n)\ll x^{\frac{1}{2}}\log^{2}x,$$ and further he could establish that 
$$\mathlarger{\mathlarger{\sum}}_{n\leq{x}}\lambda^{2}_{sym^{2}f}(n)=cx+O(x^{\theta}),$$ where $\theta<1$.\\
For more related results, see \cite{tang2013estimates}, \cite{lao2012fourth}, \cite{He2019IntegralPS} and \cite{lau2011sums}. \\ 

In 2013, Zhai \cite{zhai2013average} proved an asymptotic formula for $$\mathlarger{\mathlarger{\sum}}_{\stackrel{a^{2}+b^{2}\leq {x}}{(a,b)\in\mathbb{Z}^{2}}}\lambda^{l}_{f}(a^{2}+b^{2}),$$ 
for $x\geq1$, and $3\leq l\leq8$. For results related to mean square and higher moments of the coefficients of symmetric square $L$-functions on a certain sequence of positive integers, see \cite{anubhav2022-1} and \cite{anubhav2022-2}. \\

In an earlier paper \cite{anubhav2022-3}, we considered the sum $$S^{*}:=\mathlarger{\mathlarger{\sum}}_{\stackrel{a_{1}^{2}+a_{2}^{2}+a_{3}^{2}+a_{4}^{2}+a_{5}^{2}+a_{6}^{2}\leq {x}}{(a_{1},a_{2},a_{3},a_{4},a_{5},a_{6})\in\mathbb{Z}^{6}}}\lambda^{2}_{sym^{2}f}(a_{1}^{2}+a_{2}^{2}+a_{3}^{2}+a_{4}^{2}+a_{5}^{2}+a_{6}^{2}),$$ for sufficiently large $x$ and established the asymptotic formula:
\begin{equation}
	S^{*}=c^{*}x^{3}+O(x^{\frac{14}{5}+\epsilon}), 
\end{equation}
where $c^{*}$ is an effective constant (see Theorem 1.1 of \cite{anubhav2022-3}).\\

The main aim of this paper is to generalize and improve the result obtained in \cite{anubhav2022-3} by using recent celebrated work (\cite{Newton-thorne-1, Newton-thorne-2}) of Newton and Thorne, and better subconvexity bounds for the related $L$-functions. More precisely, we prove		

	\begin{theorem}
	Let $j\geq2$ be any fixed integer. For sufficiently large $x$, and $\epsilon>0$ any small constant, we have
	$$\mathlarger{\mathlarger{\sum}}_{\stackrel{a_{1}^{2}+a_{2}^{2}+a_{3}^{2}+a_{4}^{2}+a_{5}^{2}+a_{6}^{2}\leq {x}}{(a_{1},a_{2},a_{3},a_{4},a_{5},a_{6})\in\mathbb{Z}^{6}}}\lambda^{2}_{sym^{j}f}(a_{1}^{2}+a_{2}^{2}+a_{3}^{2}+a_{4}^{2}+a_{5}^{2}+a_{6}^{2})=c(j)x^{3}+O\left(x^{3-\frac{6}{3(j+1)^{2}+1}+\epsilon}\right) ,$$ where $c(j)$ is an effective constant defined by 
	$$c(j)=\frac{16}{3}L(3,\chi)\prod\limits_{n=1}^{j}L(1,sym^{2n}f)L(3,sym^{2n}f\otimes\chi)H_{j}(3),$$ and $\chi$ is the non-principal Dirichlet character modulo $4$.
\end{theorem}

\begin{xrem}
	\normalfont{
		When $j=2$, Theorem 1.1 gives the error term $O(x^{\frac{39}{14}+\epsilon})$, which improves the error term in (1.3).  }
\end{xrem}

	\section{Preliminaries and some important lemmas}

Let $r_{k}(n):=\#\{(n_{1},n_{2},\dots,n_{k})\in\mathbb{Z}^{k}:n_{1}^{2}+n_{2}^{2}+\cdots+n_{k}^{2}=n\}$ allowing zeros, distinguishing signs, and order.
We will be concerned with the function $r_{6}(n)$.

\begin{lemma}
	For any positive integer $n$, we have
	\begin{eqnarray}
		r_{6}(n)=16\mathlarger{\mathlarger{\sum}}_{d|n}\chi(d')d^{2}-4\mathlarger{\mathlarger{\sum}}_{d|n}\chi(d)d^{2},
	\end{eqnarray}
	where $dd'=n$, and $\chi$ is the non-principal Dirichlet character modulo $4$, i.e.,
	$$\chi(n)=\begin{cases}
		1 & \text{if $n\equiv 1\pmod4$}\\    
		-1 & \text{if $n\equiv -1\pmod4$}\\
		0 & \text{if $n\equiv 0\pmod2$}
	\end{cases}.
	$$
\end{lemma}
\begin{proof}
	See, for instance, Lemma 2.1 of \cite{anubhav2022-3}.
\end{proof}

	We can reframe (2.1) as
\begin{align*}r_{6}(n)&=16\mathlarger{\mathlarger{\sum}}_{d|n}\chi(d)\frac{n^{2}}{d^{2}}-4\mathlarger{\mathlarger{\sum}}_{d|n}\chi(d)d^{2}\\
	&=16l(n)-4v(n)&&(\mbox{say})\\
	&=l_{1}(n)-v_{1}(n) ,
\end{align*}
where $l_{1}(n)=16l(n)$ and $v_{1}(n)=4v(n)$.\\

Observe that $l(n)$ is a multiplicative function since $\chi(d)$ is completely multiplicative and hence multiplicative and $f(d)=\frac{n^{2}}{d^{2}}$ is also completely multiplicative. Therefore, $\chi(d)\dfrac{n^{2}}{d^{2}}$ is multiplicative and if $g(d)$ is any multiplicative function, then $\sum\limits_{d|n}^{} g(d)$ is also multiplicative. Similarly, $v(n)$ is also multiplicative.\\

Note that,
$$l(p)=p^{2}+\chi(p),$$
$$l(p^{2})=p^{4}+p^{2}\chi(p)+\chi(p^{2}),$$ and
$$v(p)=1+p^{2}\chi(p),$$
$$v(p^{2})=1+p^{2}\chi(p)+p^{4}\chi(p^{2}).$$

Also, we can write
\begin{align}
	&\mathlarger{\mathlarger{\sum}}_{\stackrel{a_{1}^{2}+a_{2}^{2}+a_{3}^{2}+a_{4}^{2}+a_{5}^{2}+a_{6}^{2}\leq {x}}{(a_{1},a_{2},a_{3},a_{4},a_{5},a_{6})\in\mathbb{Z}^{6}}}\lambda^{2}_{sym^{2}f}(a_{1}^{2}+a_{2}^{2}+a_{3}^{2}+a_{4}^{2}+a_{5}^{2}+a_{6}^{2}) \nonumber\\ &=\mathlarger{\mathlarger{\sum}}_{n\leq{x}}\lambda^{2}_{sym^{2}f}(n)\mathlarger{\mathlarger{\sum}}_{\stackrel{n=a_{1}^{2}+a_{2}^{2}+a_{3}^{2}+a_{4}^{2}+a_{5}^{2}+a_{6}^{2}}{(a_{1},a_{2},a_{3},a_{4},a_{5},a_{6})\in\mathbb{Z}^{6}}}1 \nonumber\\
	& =\mathlarger{\mathlarger{\sum}}_{n\leq{x}}\lambda^{2}_{sym^{2}f}(n)r_{6}(n) \nonumber\\
	& =\mathlarger{\mathlarger{\sum}}_{n\leq{x}}\lambda^{2}_{sym^{2}f}(n)\left(l_{1}(n)-v_{1}(n)\right) \\ 
	& =16\mathlarger{\mathlarger{\sum}}_{n\leq{x}}\lambda^{2}_{sym^{2}f}(n)l(n)-4\mathlarger{\mathlarger{\sum}}_{n\leq{x}}\lambda^{2}_{sym^{2}f}(n)v(n), \nonumber
\end{align}
where
$l(n)=\mathlarger{\mathlarger{\sum}}_{d|n}\chi(d)\frac{n^{2}}{d^{2}}$, and $v(n)=\mathlarger{\mathlarger{\sum}}_{d|n}\chi(d)d^{2}$.\\\\

For $j\geq2$ (an integer), and $\Re(s)>1$, the $j^{th}$ symmetric power $L$-function of $f$ is defined as :
\begin{align}
	L(s,sym^{j}f) & :=\mathlarger{\mathlarger{\sum}}_{n=1}^{\infty}\dfrac{\lambda_{sym^{j}f}(n)}{n^{s}}  \nonumber \\
	& =\prod_{p}^{}\prod_{i=0}^{j} \left(1-\frac{\alpha^{j-i}(p)\beta^{i}(p)}{p^{s}}\right)^{-1}  .
\end{align}

Observe that,
\begin{eqnarray}
	\lambda_{sym^{j}f}(p)=\mathlarger{\mathlarger{\sum}}_{m=0}^{j}\alpha^{j-m}(p)\beta^{m}(p),
\end{eqnarray}
Since $\lambda_{sym^{j}f}(n)$ is a multiplicative function, and $|\lambda_{sym^{j}f}(n)|\leq d_{j+1}(n)$ (from (1.1), and (1.2)), where $d_{j+1}(n)$ is the number of ways of expressing $n$ as a product of $j+1$ factors), so we can write the Euler product of $	L(s,sym^{j}f)$ as
\begin{equation}
	\prod_{p}^{} \left(1+\frac{\lambda_{sym^{j}f}(p)}{p^{s}}+\cdots+\frac{\lambda_{sym^{j}f}(p^{l})}{p^{ls}}+\cdots\right).
\end{equation}
Comparing (2.3), and (2.5), we get (2.4).\\
Also, due to Hecke
\begin{eqnarray}
	\lambda_{sym^{j}f}(p)=\lambda_{f}(p^{j}).
\end{eqnarray}
Further, note that
\begin{eqnarray}
	\lambda_{f}^{2}(p^{j})=1+\mathlarger{\mathlarger{\sum}}_{l=1}^{j}\lambda_{f}(p^{2l}),
\end{eqnarray}
since,

\begin{align*}
	\lambda_{f}^{2}(p^{j})=&\left(\mathlarger{\mathlarger{\sum}}_{m=0}^{j}\alpha^{j-m}(p)\beta^{m}(p)\right)^{2}\\ =&\left(\mathlarger{\mathlarger{\sum}}_{m=0}^{j}\alpha^{j-m}(p)\beta^{m}(p)\right)\left(\mathlarger{\mathlarger{\sum}}_{m^{\prime}=0}^{j}\alpha^{j-m^{\prime}}(p)\beta^{m^{\prime}}(p)\right)\\
	=&\mathlarger{\mathlarger{\sum}}_{m=0}^{j}\mathlarger{\mathlarger{\sum}}_{m^{\prime}=0}^{j}\left(\alpha^{2j-(m+m^{\prime})}(p)\right)\left(\beta^{(m+m^{\prime})}(p)\right)\\
	&(\mbox{We put $ m+m^{\prime}=t$. Observe that for every fixed integer $t$ in the interval $[0,2l]$} \\ 
	& \mbox{and for every fixed integer $m$ in the interval $[0,l]$, there is a unique integer $m^{\prime}$  } \\ 
	&\mbox{in the interval $[0,l]$ satisfying $ m+m^{\prime}=t$ and thus,})\\
	=&\mathlarger{\mathlarger{\sum}}_{l=0}^{j}\left(\mathlarger{\mathlarger{\sum}}_{t=0}^{2l}\alpha^{2j-t}(p)\beta^{t}(p)\right)        \\   
	=&1+\mathlarger{\mathlarger{\sum}}_{l=1}^{j}\left(\mathlarger{\mathlarger{\sum}}_{t=0}^{2l}\alpha^{2j-t}(p)\beta^{t}(p)\right) \\
	=&1+\mathlarger{\mathlarger{\sum}}_{l=1}^{j}\lambda_{f}(p^{2l}).
\end{align*}

\begin{lemma}
	Let $f$ be a normalized primitive holomorphic cusp form of weight $k$ for $SL(2,\mathbb{Z})$, and let $\lambda_{sym^{j}f}(n)$ be the $n^{th}$ normalized Fourier coefficient of the $j^{th}$ symmetric power $L$-function associated to $f$. If
	$$F_{j}(s)=\mathlarger{\mathlarger{\sum}}_{n=1}^{\infty}\dfrac{\lambda^{2}_{sym^{j}f}(n)l(n)}{n^{s}} ,$$ for $\Re(s)>3$, then
	$$F_{j}(s)=G_{j}(s)H_{j}(s),$$ where
	$$
	G_{j}(s):=\zeta(s-2)L(s,\chi)\prod\limits_{n=1}^{j}L(s-2,sym^{2n}f)L(s,sym^{2n}f\otimes\chi) ,
	$$ 
	and $\chi$ is the non-principal character modulo $4$. \\
	Here, $H_{j}(s)$ is a Dirichlet series which converges uniformly, and absolutely in the half plane $\Re(s)>\frac{5}{2}$, and $H_{j}(s)\ne 0$ on $\Re(s)=3$.
\end{lemma}

\begin{proof}
	We observe that $\lambda^{2}_{sym^{j}f}(n)l(n)$ is multiplicative and hence
	\begin{equation}
		F_{j}(s)=\prod_{p}^{} \left(1+\frac{\lambda^{2}_{sym^{j}f}(p)l(p)}{p^{s}}+\cdots+\frac{\lambda^{2}_{sym^{j}f}(p^{m})l(p^{m})}{p^{ms}}+\cdots\right).
	\end{equation}
	
	Using (2.6) and (2.7), we note that,
	\begin{align*}
		\lambda^{2}_{sym^{2}f}(p)l(p) & = \lambda_{f}^{2}(p^{j})\left(p^{2}+\chi(p)\right) \\
		&=\left(1+\mathlarger{\mathlarger{\sum}}_{l=1}^{j}\lambda_{f}(p^{2l})\right)\left(p^{2}+\chi(p)\right) \\
		&=\left(1+\mathlarger{\mathlarger{\sum}}_{l=1}^{j}\lambda_{sym^{2l}f}(p)\right)\left(p^{2}+\chi(p)\right) \\
		& =p^{2}+\chi(p)+\mathlarger{\mathlarger{\sum}}_{l=1}^{j}\lambda_{sym^{2l}f}(p)p^{2}+\mathlarger{\mathlarger{\sum}}_{l=1}^{j}\lambda_{sym^{2l}f}(p)\chi(p) \\
		& =:b(p) .             &&(\mbox{say})
	\end{align*} 
	
	From the structure of $b(p)$, we define the coefficients $b(n)$ as
	
	\begin{align*}
		\mathlarger{\mathlarger{\sum}}_{n=1}^{\infty}\frac{b(n)}{n^{s}}=&\zeta(s-2)L(s,\chi)\prod\limits_{n=1}^{j}L(s-2,sym^{2n}f)L(s,sym^{2n}f\otimes\chi),
	\end{align*}
	which is absolutely convergent in $\Re(s)>3$. We also note that,
	\begin{align*}
		&\prod_{p}^{} \left(1+\frac{b(p)}{p^{s}}+\cdots+\frac{b(p^{m})}{p^{ms}}+\cdots\right) \\
		&=\zeta(s-2)L(s,\chi)\prod\limits_{n=1}^{j}L(s-2,sym^{2n}f)L(s,sym^{2n}f\otimes\chi) \\
		&=G_{j}(s), &&(\mbox{say})
	\end{align*}
	for $\Re(s)>3$. Observe that $b(n)\ll_{\epsilon}n^{2+\epsilon}$ for any small positive constant $\epsilon$. \\
	
	Now, we note that
	\begin{align*}
		\Bigg|\frac{b(p)}{p^{s}}+\frac{b(p^{2})}{p^{2s}}+\cdots+\frac{b(p^{m})}{p^{ms}}+\cdots\Bigg|&\ll\mathlarger{\mathlarger{\sum}}_{m=1}^{\infty}\frac{p^{(2+\epsilon)m}}{p^{m\sigma}} \\
		&\leq \mathlarger{\mathlarger{\sum}}_{m=1}^{\infty}\frac{p^{(2+\epsilon)m}}{p^{(3+2\epsilon)m}} &&(\mbox{in $\Re(s)\geq 3+2\epsilon$}) \\
		&=\mathlarger{\mathlarger{\sum}}_{m=1}^{\infty}\frac{1}{p^{(1+\epsilon)m}} \\
		&=\dfrac{\dfrac{1}{p^{1+\epsilon}}}{1-\dfrac{1}{p^{1+\epsilon}}} \\
		&=\dfrac{1}{p^{1+\epsilon}-1}\\
		&<1.
	\end{align*}
	
	Let us write 
	$$A=\frac{\lambda^{2}_{sym^{j}f}(p)l(p)}{p^{s}}+\cdots+\frac{\lambda^{2}_{sym^{j}f}(p^{m})l(p^{m})}{p^{ms}}+\cdots,$$ and 
	$$B=\frac{b(p)}{p^{s}}+\cdots+\frac{b(p^{m})}{p^{ms}}+\cdots.$$
	From the above calculations, we observe that $|B|<1$ in $\Re(s)\geq 3+2\epsilon$. \\\\
	
	Note that,
	\begin{align*}
		\frac{1+A}{1+B}&=(1+A)(1-B+B^{2}-B^{3}+\cdots) &&(\mbox{in $\Re(s)\geq 3+2\epsilon$}) \\
		&=1+A-B-AB+ \mbox{higher terms} \\
		&=1+\frac{\lambda^{2}_{sym^{j}f}(p^{2})l(p^{2})-b(p^{2})}{p^{2s}}+\cdots+\frac{c_{m}(p^{m})}{p^{ms}}+\cdots,
	\end{align*}
	with $c_{m}(n)\ll_{\epsilon}n^{2+\epsilon}$. So,
	\begin{align*}
		\prod_{p}^{}\left(\frac{1+A}{1+B}\right)&=	\prod_{p}^{}\left(1+\frac{\lambda^{2}_{sym^{j}f}(p^{2})l(p^{2})-b(p^{2})}{p^{2s}}+\cdots+\frac{c_{m}(p^{m})}{p^{ms}}+\cdots\right) \\
		&\ll_{\epsilon} 1 . \hspace{6cm} (\mbox{in $\Re(s)>\frac{5}{2}$})
	\end{align*} 
	
	Thus,
	\begin{align*}
		H_{j}(s)&:=\frac{F_{j}(s)}{G_{j}(s)} \\
		&=\prod_{p}^{}\left(\frac{1+A}{1+B}\right) \\
		&\ll_{\epsilon} 1, &&(\mbox{in $\Re(s)>\frac{5}{2}$})
	\end{align*} 
	and also $H_{j}(s)\neq0$ on $\Re(s)=3$. 
	
\end{proof}

\begin{lemma}
	Let $f$ be a normalized primitive holomorphic cusp form of weight $k$ for $SL(2,\mathbb{Z})$, and let $\lambda_{sym^{j}f}(n)$ be the $n^{th}$ normalized Fourier coefficient of the $j^{th}$ symmetric power $L$-function associated to $f$. If
	$$\widetilde{F}_{j}(s)=\mathlarger{\mathlarger{\sum}}_{n=1}^{\infty}\dfrac{\lambda^{2}_{sym^{j}f}(n)v(n)}{n^{s}} ,$$ for $\Re(s)>3$, then
	$$\widetilde{F}_{j}(s)=\widetilde{G}_{j}(s)\widetilde{H}_{j}(s),$$ where
	\begin{align*}
		\widetilde{G}_{2}(s):=&\zeta(s)L(s-2,\chi)\prod\limits_{n=1}^{j}L(s,sym^{2n}f)L(s-2,sym^{2n}f\otimes\chi) ,
	\end{align*} and $\chi$ is the non-principal character modulo $4$. \\
	Here $\widetilde{H}_{j}(s)$ is a Dirichlet series which converges uniformly, and absolutely in the half plane $\Re(s)>\frac{5}{2}$, and $\widetilde{H}_{j}(s)\ne 0$ on $\Re(s)=3$.
\end{lemma}
\begin{proof}
	We observe that $\lambda^{2}_{sym^{j}f}(n)v(n)$ is multiplicative, and hence
	\begin{equation}
		\widetilde{F}_{2}(s)=\prod_{p}^{} \left(1+\frac{\lambda^{2}_{sym^{j}f}(p)v(p)}{p^{s}}+\cdots+\frac{\lambda^{2}_{sym^{j}f}(p^{m})v(p^{m})}{p^{ms}}+\cdots\right).
	\end{equation}
	
	Using (2.6) and (2.7), we note that,
	\begin{align*}
		\lambda^{2}_{sym^{2}f}(p)v(p) & = \lambda_{f}^{2}(p^{j})\left(1+p^{2}\chi(p)\right) \\
		&=\left(1+\mathlarger{\mathlarger{\sum}}_{l=1}^{j}\lambda_{f}(p^{2l})\right)\left(1+p^{2}\chi(p)\right) \\
		&=\left(1+\mathlarger{\mathlarger{\sum}}_{l=1}^{j}\lambda_{sym^{2l}f}(p)\right)\left(1+p^{2}\chi(p)\right) \\
		& =1+p^{2}\chi(p)+\mathlarger{\mathlarger{\sum}}_{l=1}^{j}\lambda_{sym^{2l}f}(p)+\mathlarger{\mathlarger{\sum}}_{l=1}^{j}\lambda_{sym^{2l}f}(p)p^{2}\chi(p) \\
		& =:h(p) .             &&(\mbox{say})
	\end{align*} 
	
	From the structure of $h(p)$, we define the coefficients $h(n)$ as
	
	\begin{align*}
		\mathlarger{\mathlarger{\sum}}_{n=1}^{\infty}\frac{h(n)}{n^{s}}=&\zeta(s)L(s-2,\chi)\prod\limits_{n=1}^{j}L(s,sym^{2n}f)L(s-2,sym^{2n}f\otimes\chi),
	\end{align*}
	which is absolutely convergent in $\Re(s)>3$. We also note that, 
	\begin{align*}
		&\prod_{p}^{} \left(1+\frac{h(p)}{p^{s}}+\cdots+\frac{h(p^{m})}{p^{ms}}+\cdots\right) \\
		&=\zeta(s)L(s-2,\chi)\prod\limits_{n=1}^{j}L(s,sym^{2n}f)L(s-2,sym^{2n}f\otimes\chi) \\
		&=\widetilde{G}_{2}(s) , \hspace{10cm} \text{(say)}
	\end{align*}
	for $\Re(s)>3$. Observe that $h(n)\ll_{\epsilon}n^{2+\epsilon}$ for any small positive constant $\epsilon$. \\\\
	
	Now, we note that
	\begin{align*}
		\Bigg|\frac{h(p)}{p^{s}}+\frac{h(p^{2})}{p^{2s}}+\cdots+\frac{h(p^{m})}{p^{ms}}+\cdots\Bigg|&\ll\mathlarger{\mathlarger{\sum}}_{m=1}^{\infty}\frac{p^{(2+\epsilon)m}}{p^{m\sigma}} \\
		&<1. &&(\mbox{in $\Re(s)\geq 3+2\epsilon$})
	\end{align*}
	
	Let us write,
	$$\widetilde{A}=\frac{\lambda^{2}_{sym^{j}f}(p)v(p)}{p^{s}}+\cdots+\frac{\lambda^{2}_{sym^{j}f}(p^{m})v(p^{m})}{p^{ms}}+\cdots,$$ and 
	$$\widetilde{B}=\frac{h(p)}{p^{s}}+\cdots+\frac{h(p^{m})}{p^{ms}}+\cdots.$$
	From the above calculations, we observe that $|\widetilde{B}|<1$ in $\Re(s)\geq 3+2\epsilon$. \\\\
	
	Note that,
	\begin{align*}
		\frac{1+\widetilde{A}}{1+\widetilde{B}}&=(1+\widetilde{A})(1-\widetilde{B}+\widetilde{B}^{2}-\widetilde{B}^{3}+\cdots) &&(\mbox{in $\Re(s)\geq 3+2\epsilon$}) \\
		&=1+\widetilde{A}-\widetilde{B}-\widetilde{A}\widetilde{B}+ \mbox{higher terms} \\
		&=1+\frac{\lambda^{2}_{sym^{j}f}(p^{2})v(p^{2})-h(p^{2})}{p^{2s}}+\cdots+\frac{\widetilde{c}_{m}(p^{m})}{p^{ms}}+\cdots,
	\end{align*}
	with $\widetilde{c}_{m}(n)\ll_{\epsilon}n^{2+\epsilon}$. So,
	\begin{align*}
		\prod_{p}^{}\left(\frac{1+\widetilde{A}}{1+\widetilde{B}}\right)&=	\prod_{p}^{}\left(1+\frac{\lambda^{2}_{sym^{j}f}(p^{2})v(p^{2})-h(p^{2})}{p^{2s}}+\cdots+\frac{\widetilde{c}_{m}(p^{m})}{p^{ms}}+\cdots\right) \\
		&\ll_{\epsilon} 1. \hspace{6cm} (\mbox{in $\Re(s)>\frac{5}{2}$})
	\end{align*}
	
	Thus,
	\begin{align*}
		\widetilde{H}_{j}(s)&:=\frac{\widetilde{F}_{j}(s)}{\widetilde{G}_{j}(s)} \\
		&=\prod_{p}^{}\left(\frac{1+\widetilde{A}}{1+\widetilde{B}}\right) \\
		&\ll_{\epsilon} 1, &&(\mbox{in $\Re(s)>\frac{5}{2}$})
	\end{align*} 
	and also $\widetilde{H}_{j}(s)\neq0$ on $\Re(s)=3$.
\end{proof}

\begin{lemma}
	For any $\epsilon>0$, we have
	\begin{eqnarray}
		\int_{1}^{T} \left|L \left( \dfrac{1}{2}+it,f \right) \right|^{6}dt \ll T^{2+\epsilon} ,
	\end{eqnarray}
	uniformly for $T\geq1 $, and
	\begin{eqnarray}
		L(\sigma+it) \ll_{\epsilon} (1+|t|)^{\frac{1}{3}(1+\epsilon-\sigma)+\epsilon},
	\end{eqnarray} 
	uniformly for $\frac{1}{2}\leq \sigma\leq 1+\epsilon$, and $|t|\geq t_{0}$ (where $t_0$ is sufficiently large).
\end{lemma}
\begin{proof}
	Proof of (2.10) is given by Jutila, for instance, see \cite{jutila1987} and using maximum-modulus principle in a suitable rectangle, we get (2.11), for instance, see \cite{good1982square}.
\end{proof}

\begin{lemma}
	For any $\epsilon>0$, we have
	\begin{eqnarray}
		\int_{1}^{T} \left|\zeta \left( \dfrac{1}{2}+it \right) \right|^{12}dt \ll T^{2+\epsilon} ,
	\end{eqnarray}
	uniformly for $T\geq1 $, and
	
	\begin{eqnarray}
		\zeta(\sigma+it) \ll_{\epsilon} (1+|t|)^{\max\{\frac{13}{42}(1-\sigma),0\}+\epsilon},
	\end{eqnarray} 
	uniformly for $\frac{1}{2}\leq\sigma\leq1+\epsilon$, and $|t|\geq1$.
\end{lemma}

\begin{proof}
	For the proof of (2.12), see \cite{heath-brown-12th} and (2.13) is due to Bourgain, for instance, see \cite{bourgain-decoupling}.
\end{proof}

\begin{lemma}
	Let $\chi$ be a primitive character modulo $q$ and $\mathfrak{L}^{d}_{m,n}(s,\chi)$  be a general $L$-function of degree $2A$. For any $\epsilon>0$, we have 
	\begin{eqnarray}
		\int_{T}^{2T} \left|\mathfrak{L}^{d}_{m,n}(\sigma+it,\chi) \right|^{2}dt \ll (qT)^{2A(1-\sigma)+\epsilon} ,
	\end{eqnarray}
	uniformly for $\frac{1}{2}\leq\sigma\leq1+\epsilon$, and $T\geq1$. Also,
	\begin{eqnarray}
		\mathfrak{L}^{d}_{m,n}(\sigma+it,\chi)  \ll \left(q(1+|t|)\right)^{\max\{A(1-\sigma),0\}+\epsilon},  
	\end{eqnarray}
	uniformly for $-\epsilon\leq\sigma\leq1+\epsilon$.
\end{lemma}

\begin{proof}
	For the proof of (2.14) and (2.15), see \cite{jiang-lu}.
	
\end{proof}

\section{Proof of Theorem 1.1}	

From (2.2), we can write
$$\mathlarger{\mathlarger{\sum}}_{n\leq{x}}\lambda^{2}_{sym^{j}f}(n)r_{6}(n)=\mathlarger{\mathlarger{\sum}}_{n\leq{x}}\lambda^{2}_{sym^{j}f}(n)l_{1}(n)-\mathlarger{\mathlarger{\sum}}_{n\leq{x}}\lambda^{2}_{sym^{j}f}(n)v_{1}(n).$$

Firstly, we consider the sum $\mathlarger{\mathlarger{\sum}}_{n\leq{x}}\lambda^{2}_{sym^{j}f}(n)l_{1}(n)$. We begin by applying the Perron's formula (see chapter 2.4 of \cite{granville2014multiplicative}) to $F_{j}(s)$ with $\eta=3+\epsilon$, and\\ $10\leq T\leq x$. Thus we have,
\begin{align*}
	\mathlarger{\mathlarger{\sum}}_{n\leq{x}}\lambda^{2}_{sym^{j}f}(n)l_{1}(n)  & =16\mathlarger{\mathlarger{\sum}}_{n\leq{x}}\lambda^{2}_{sym^{j}f}(n)l(n) \\
	& =\frac{16}{2\pi i}\int_{\eta-iT}^{\eta+iT}F_{j}(s)\frac{x^s}{s}ds + O\left(\frac{x^{3+3\epsilon}}{T}\right).
\end{align*}
We move the line of integration to $\Re(s)=\frac{5}{2}+\epsilon$, and by Cauchy's residue theorem there is only one simple pole at $s=3$, coming from the factor $\zeta(s-2)$. This contributes a residue, which is $c(j)x^{3}$, where $c(j)$ is an effective constant depending on the values of various $L$-functions appearing in $G_{j}(s)$ at $s=3$.  \\

More precisely,
\begin{align*}
	c(j) &=16\lim_{s \to 3} (s-3)\frac{F_{j}(s)}{s} \\
	&= \frac{16}{3}L(3,\chi)\prod\limits_{n=1}^{j}L(1,sym^{2n}f)L(3,sym^{2n}f\otimes\chi)H_{j}(3).
\end{align*}
So, we obtain
\begin{align*}
	\mathlarger{\mathlarger{\sum}}_{n\leq{x}}\lambda^{2}_{sym^{j}f}(n)l_{1}(n) & = c(j)x^{3}+\frac{16}{2\pi i}\left\{\int_{\frac{5}{2}+\epsilon-iT}^{\frac{5}{2}+\epsilon+iT}+\int_{3+\epsilon-iT}^{\frac{5}{2}+\epsilon-iT}+\int_{\frac{5}{2}+\epsilon+iT}^{3+\epsilon+iT}\right\}F_{j}(s)\frac{x^{s}}{s}ds\\ & \quad +O\left(\frac{x^{3+3\epsilon}}{T}\right) \\
	& =c(j)x^{3} + \frac{16}{2\pi i}(J_{1}+J_{2}+J_{3}) +O\left(\frac{x^{3+3\epsilon}}{T}\right). &&(\mbox{say})
\end{align*}

Contribution of horizontal line integrals ($J_{2}$ and $J_{3}$) in absolute value (using Lemmas 2.2, 2.5 and 2.6) is
\begin{align*}
	& \ll \int_{\frac{5}{2}+\epsilon}^{3+\epsilon}\frac{|\zeta(\sigma-2+iT)\prod\limits_{n=1}^{j}L(\sigma-2+iT,sym^{2n}f)|}{T}x^{\sigma}d\sigma \\\\
	& \ll \int_{\frac{1}{2}+\epsilon}^{1+\epsilon}\frac{|\zeta(\sigma+iT)\prod\limits_{n=1}^{j}L(\sigma+iT,sym^{2n}f)|}{T}x^{\sigma+2}d\sigma \\\\
	& \ll \left(\frac{x^{2}}{T}\right)\underset{\frac{1}{2}+\epsilon\leq{\sigma}\leq{1+\epsilon}}{\max}\ \ x^\sigma T^{\left(\frac{(j+1)^{2}}{2}-\frac{4}{21}\right)(1-\sigma)+\epsilon}  \\
	& \ll \left(\frac{x^{2+2\epsilon}}{T}\right)\underset{\frac{1}{2}+\epsilon\leq{\sigma}\leq{1+\epsilon}}{\max}\left( \frac{x}{T^{\left(\frac{(j+1)^{2}}{2}-\frac{4}{21}\right)}}\right)^\sigma T^{\left(\frac{(j+1)^{2}}{2}-\frac{4}{21}\right)} .
\end{align*}

Clearly, $\left( \dfrac{x}{T^{\left(\frac{(j+1)^{2}}{2}-\frac{4}{21}\right)}}\right)^\sigma$ is monotonic as a function of $\sigma$ for $\frac{1}{2}+\epsilon\leq{\sigma}\leq{1+\epsilon}$, and hence the maximum is attained at the extremities of the interval $[\frac{1}{2}+\epsilon,1+\epsilon]$. Thus, \\
$\ll x^{2+2\epsilon}\left(x^{\frac{1}{2}+\epsilon}T^{\left(\frac{(j+1)^{2}}{4}-\frac{2}{21}-1\right)}+\frac{x^{1+\epsilon}}{T}\right)$ \\\\ 
$\ll x^{\frac{5}{2}+3\epsilon}T^{\left(\frac{(j+1)^{2}}{4}-\frac{23}{21}+\epsilon\right)} +\frac{x^{3+3\epsilon}}{T} $. \\\\

Contribution of the left vertical line integral ($J_{1}$) in absolute value (using Lemmas 2.2, 2.5, 2.6 and H\"{o}lder's inequality) is \\
\begin{align*}
	& \ll \int_{\frac{5}{2}+\epsilon-iT}^{\frac{5}{2}+\epsilon+iT}\frac{\left|\zeta(\frac{1}{2}+\epsilon+it)\prod\limits_{n=1}^{j}L(\frac{1}{2}+\epsilon+it,sym^{2n}f)\right|}{\left|\frac{5}{2}+\epsilon+it\right|}x^{\frac{5}{2}+\epsilon}dt \\\\
	& \ll x^{\frac{5}{2}+\epsilon}+x^{\frac{5}{2}+\epsilon}\frac{1}{T}\left(\int_{10\leq|t|\leq T}^{} \left|\zeta(\frac{1}{2}+\epsilon+it)\right|^{12}dt\right)^\frac{1}{12}\\
	& \quad \times \left(\int_{10\leq|t|\leq T}^{}\left|L(\frac{1}{2}+\epsilon+it,sym^{2}f)\right|^2dt\right)^\frac{1}{2} \\\\
	& \quad \times \left\{ \underset{10\leq{|t|}\leq{T}}{\max}\left|\prod\limits_{n=2}^{j}L(\frac{1}{2}+\epsilon+it,sym^{2n}f)\right|^{\frac{2}{5}}\left(\int_{10\leq|t|\leq T}^{}\left|\prod\limits_{n=2}^{j}L(\frac{1}{2}+\epsilon+it,sym^{2n}f)\right|^2dt\right)\right\}^\frac{5}{12} \\\\
	& \ll x^{\frac{5}{2}+\epsilon}+x^{\frac{5}{2}+\epsilon}\left(T^{-1+2.\frac{1}{12}+3.\frac{1}{2}.\frac{1}{2}+\left((j+1)^{2}-4\right)\left(\frac{1}{2}.\frac{1}{2}.\frac{2}{5}.\frac{5}{12}\right)+\left((j+1)^{2}-4\right)\left(\frac{1}{2}.\frac{5}{12}\right)}\right) \\\\
	& \ll x^{\frac{5}{2}+\epsilon}T^{\left(\frac{(j+1)^{2}}{4}-\frac{13}{12}+\epsilon\right)} .
\end{align*}

Note that $10\leq T\leq x$. Thus, we obtain
$$\mathlarger{\mathlarger{\sum}}_{n\leq{x}}\lambda^{2}_{sym^{j}f}(n)l_{1}(n) = c(j)x^{3}+O(x^{\frac{5}{2}+\epsilon}T^{\left(\frac{(j+1)^{2}}{4}-\frac{13}{12}+\epsilon\right)})+O\left(\frac{x^{3+3\epsilon}}{T}\right). $$
We choose $T$ such that $x^{\frac{5}{2}}T^{\left(\frac{(j+1)^{2}}{4}-\frac{13}{12}\right)}\asymp \frac{x^{3}}{T} $ i.e., $T^{\left(\frac{3(j+1)^{2}-1}{12}\right)}\asymp x^{\frac{1}{2}}$.\\\\ Therefore, $T\asymp x^{\frac{6}{3(j+1)^{2}-1}}$. \\\\

Thus, we get
\begin{eqnarray}
	\mathlarger{\mathlarger{\sum}}_{n\leq{x}}\lambda^{2}_{sym^{j}f}(n)l_{1}(n)=c(j)x^{3}+O\left(x^{3-\frac{6}{3(j+1)^{2}-1}+3\epsilon}\right).
\end{eqnarray}

Similarly, we apply the Perron's formula (see chapter 2.4 of \cite{granville2014multiplicative}) to $\widetilde{F}_{j}(s)$ with $\eta=3+\epsilon$, and $10\leq T\leq x$. Thus, we have
\begin{align*}
	\mathlarger{\mathlarger{\sum}}_{n\leq{x}}\lambda^{2}_{sym^{j}f}(n)v_{1}(n)  & =4\mathlarger{\mathlarger{\sum}}_{n\leq{x}}\lambda^{2}_{sym^{j}f}(n)v(n) \\
	& =\frac{4}{2\pi i}\int_{\eta-iT}^{\eta+iT}\widetilde{F}_{j}(s)\frac{x^s}{s}ds + O\left(\frac{x^{3+3\epsilon}}{T}\right).
\end{align*}
We move the line of integration to $\Re(s)=\frac{5}{2}+\epsilon$. Note that, there is no singularity in the rectangle obtained, and the function $\widetilde{F}_{j}(s)\dfrac{x^s}{s}$ is analytic in this region. Thus, using Cauchy's theorem for rectangle pertaining to analytic functions, we get
\begin{align*}
	\mathlarger{\mathlarger{\sum}}_{n\leq{x}}\lambda^{2}_{sym^{j}f}(n)v_{1}(n) & = \frac{4}{2\pi i}\left\{\int_{\frac{5}{2}+\epsilon-iT}^{\frac{5}{2}+\epsilon+iT}+\int_{3+\epsilon-iT}^{\frac{5}{2}+\epsilon-iT}+\int_{\frac{5}{2}+\epsilon+iT}^{3+\epsilon+iT}\right\}\widetilde{F}_{j}(s)\frac{x^{s}}{s}ds\\ & \quad +O\left(\frac{x^{3+3\epsilon}}{T}\right) \\
	& =\frac{4}{2\pi i}(J_{1}+J_{2}+J_{3}) +O\left(\frac{x^{3+3\epsilon}}{T}\right). &&(\mbox{say})
\end{align*}

Contribution of horizontal line integrals ($J_{2}$ and $J_{3}$) in absolute value (using Lemmas 2.3, 2.4 and 2.6 ) is
\begin{align*}
	& \ll \int_{\frac{5}{2}+\epsilon}^{3+\epsilon}\frac{|L(\sigma-2+iT)\prod\limits_{n=1}^{j}L(\sigma-2+iT,sym^{2n}f\otimes\chi)|}{T}x^{\sigma}d\sigma \\\\
	& \ll \int_{\frac{1}{2}+\epsilon}^{1+\epsilon}\frac{|L(\sigma+iT)\prod\limits_{n=1}^{j}L(\sigma+iT,sym^{2n}f\otimes\chi)|}{T}x^{\sigma+2}d\sigma \\\\
	& \ll \left(\frac{x^{2}}{T}\right)\underset{\frac{1}{2}+\epsilon\leq{\sigma}\leq{1+\epsilon}}{\max}\ \ x^\sigma T^{\left(\frac{(j+1)^{2}}{2}-\frac{1}{6}\right)(1-\sigma)+\epsilon}  \\
	& \ll \left(\frac{x^{2+2\epsilon}}{T}\right)\underset{\frac{1}{2}+\epsilon\leq{\sigma}\leq{1+\epsilon}}{\max}\left( \frac{x}{T^{\left(\frac{(j+1)^{2}}{2}-\frac{1}{6}\right)}}\right)^\sigma T^{\left(\frac{(j+1)^{2}}{2}-\frac{1}{6}\right)} .
\end{align*}	

Clearly, $\left( \dfrac{x}{T^{\left(\frac{(j+1)^{2}}{2}-\frac{1}{6}\right)}}\right)^\sigma$ is monotonic as a function of $\sigma$ for $\frac{1}{2}+\epsilon\leq{\sigma}\leq{1+\epsilon}$, and hence the maximum is attained at the extremities of the interval $[\frac{1}{2}+\epsilon,1+\epsilon]$. Thus, \\
$\ll x^{2+2\epsilon}\left(x^{\frac{1}{2}+\epsilon}T^{\left(\frac{(j+1)^{2}}{4}-\frac{1}{12}-1+\epsilon\right)}+\frac{x^{1+\epsilon}}{T}\right)$ \\\\ 
$\ll x^{\frac{5}{2}+3\epsilon}T^{\left(\frac{(j+1)^{2}}{4}-\frac{13}{12}+\epsilon\right)} +\frac{x^{3+3\epsilon}}{T} $. \\\\

Contribution of the left vertical line integral ($J_{1}$) in absolute value (using Lemmas 2.3, 2.4, 2.6 and H\"{o}lder's inequality) is \\
\begin{align*}
	& \ll \int_{\frac{5}{2}+\epsilon-iT}^{\frac{5}{2}+\epsilon+iT}\frac{\left|L(\frac{1}{2}+\epsilon+it)\prod\limits_{n=1}^{j}L(\frac{1}{2}+\epsilon+it,sym^{2n}f\otimes\chi)\right|}{\left|\frac{5}{2}+\epsilon+it\right|}x^{\frac{5}{2}+\epsilon}dt \\\\
	& \ll x^{\frac{5}{2}+\epsilon}+x^{\frac{5}{2}+\epsilon}\frac{1}{T}\left(\int_{10\leq|t|\leq T}^{} \left|L(\frac{1}{2}+\epsilon+it)\right|^{6}dt\right)^\frac{1}{6}\\
	& \quad \times \left(\int_{10\leq|t|\leq T}^{}\left|L(\frac{1}{2}+\epsilon+it,sym^{2}f\otimes\chi)\right|^2dt\right)^\frac{1}{2} \\\\
	& \quad \times \left\{ \underset{10\leq{|t|}\leq{T}}{\max}\left|\prod\limits_{n=2}^{j}L(\frac{1}{2}+\epsilon+it,sym^{2n}f\otimes\chi)\right|\left(\int_{10\leq|t|\leq T}^{}\left|\prod\limits_{n=2}^{j}L(\frac{1}{2}+\epsilon+it,sym^{2n}f\otimes\chi)\right|^2dt\right)\right\}^\frac{1}{3} \\\\
	& \ll x^{\frac{5}{2}+\epsilon}+x^{\frac{5}{2}+\epsilon}\left(T^{-1+2.\frac{1}{6}+3.\frac{1}{2}.\frac{1}{2}+\left((j+1)^{2}-4\right)\left(\frac{1}{2}.\frac{1}{2}.\frac{1}{3}\right)+\left((j+1)^{2}-4\right)\left(\frac{1}{2}.\frac{1}{3}\right)}\right) \\\\
	& \ll x^{\frac{5}{2}+\epsilon}T^{\left(\frac{(j+1)^{2}}{4}-\frac{11}{12}\right)} .
\end{align*}

Note that $10\leq T\leq x$. Thus, we obtain
$$\mathlarger{\mathlarger{\sum}}_{n\leq{x}}\lambda^{2}_{sym^{j}f}(n)v_{1}(n) = O\left(x^{\frac{5}{2}+\epsilon}T^{\left(\frac{(j+1)^{2}}{4}-\frac{11}{12}\right)}\right)+O\left(\frac{x^{3+3\epsilon}}{T}\right). $$
We choose $T$ such that $x^{\frac{5}{2}}T^{\left(\frac{(j+1)^{2}}{4}-\frac{11}{12}\right)}\asymp \frac{x^{3}}{T} $ i.e., $T^{\left(\frac{(j+1)^{2}}{4}+\frac{1}{12}\right)}\asymp x^{\frac{1}{2}}$.\\\\ Therefore, $T\asymp x^{\frac{6}{3(j+1)^{2}+1}}$. \\\\

Thus, we get
\begin{eqnarray}
	\mathlarger{\mathlarger{\sum}}_{n\leq{x}}\lambda^{2}_{sym^{j}f}(n)v_{1}(n)=O\left(x^{3-\frac{6}{3(j+1)^{2}+1}+\epsilon}\right).
\end{eqnarray}
Combining (3.1) and (3.2), we get
$$\mathlarger{\mathlarger{\sum}}_{n\leq{x}}\lambda^{2}_{sym^{j}f}(n)r_{6}(n)=c(j)x^{3}+O\left(x^{3-\frac{6}{3(j+1)^{2}+1}+\epsilon}\right) ,$$ where $c(j)$ is an effective constant given by 
$$c(j)=\frac{16}{3}L(3,\chi)\prod\limits_{n=1}^{j}L(1,sym^{2n}f)L(3,sym^{2n}f\otimes\chi)H_{j}(3),$$ and $\chi$ is the non-principal Dirichlet character modulo $4$.\\
This proves the theorem.\\\\
\textbf{Concluding Remarks:} Note that we have the expected upper bounds, namely, 
$$\int_{T}^{2T} \left|\zeta \left( \dfrac{5}{7}+it \right) \right|^{12}dt \ll T^{1+\epsilon}$$
and
$$\int_{T}^{2T} \left|L \left( \dfrac{5}{8}+it,f \right) \right|^{4}dt \ll T^{1+\epsilon},$$
uniformly for $T\geq1 $ (see \cite{ivic,good1982square}). Even if we move the line of integration to $\Re(s)=\frac{5}{7}$ and $\Re(s)=\frac{5}{8}$ pertaining to $l_{1}(n)$ and $v_{1}(n)$ respectively, and using the arguments of this paper, we end up with the same error term as stated in the Theorem 1.1.

\subsection*{Acknowledgements}
The first author, Anubhav Sharma, wishes to express his gratitude to the University of Hyderabad for
its financial support and IoE's performance based publication incentive towards his Ph.D Program.


\begin{thebibliography}{HD82}
	
	
	
	
	\normalsize
	\baselineskip=17pt


    \bibitem{bourgain-decoupling}
    J. Bourgain,  {\it Decoupling, exponential sums and the Riemann zeta function}, Journal of the American Mathematical Society 30(1) (2017), 205-224. 

    \bibitem{deligne1974conjecture}
    P. Deligne, {\it La conjecture de Weil. I}, Publications Mathématiques de l'Institut des Hautes Études Scientifiques  43(1) (1974), 273-307.	  
    
    \bibitem{fomenko2006identities}
    O. M. Fomenko,  {\it Identities Involving Coefficients of Automorphic $L$-Functions}, Journal of Mathematical Sciences (2006) 133(6), 1749-1755.
    
    \bibitem{good1982square}
    A. Good,  {\it The square mean of Dirichlet series associated with cusp forms}, Mathematika 29(2) (1982), 278-295.
    
    
     \bibitem{granville2014multiplicative}
    A. Granville and K. Soundararajan,  {\it Multiplicative number theory: The pretentious approach}, Book manuscript in preparation  12 (2014).
    
   
    
    \bibitem{hardy1979introduction}
    G. H. Hardy and E. M. Wright, {\it An introduction to the theory of numbers}, Oxford university press (1979). 
    
    \bibitem{He2019IntegralPS}
    X. He, {\it Integral power sums of Fourier coefficients of symmetric square $L$-functions}, Proceedings of the American Mathematical Society  147(7) (2019), 2847-2856.
    
    \bibitem{heath-brown-12th}
    D. R. Heath-Brown,  {\it The twelfth power moment of the Riemann-function}, The Quarterly Journal of Mathematics 29(4) (1978), 443-462.
    
     \bibitem{ivic}
    A. Ivic,  {\it Exponent pairs and the zeta function of Riemann}, Studia Sci. Math. Hungar 15(1-3) (1980), 157-181.
    
    \bibitem{jiang-lu}
    Y. Jiang and G. Lü,   {\it On the higher mean over arithmetic progressions of Fourier coefficients of cusp forms}, Acta Arithmetica 3(166) (2014), 231-252.
    
     \bibitem{jutila1987}
    M. Jutila and Tata Institute of Fundamental Research (Bombay), {\it Lectures on a Method in the Theory of Exponential Sums},  80 (1987).
    
    
    \bibitem{lao2012fourth}
    H. Lao,  {\it On the fourth moment of coefficients of symmetric square $L$-function}, Chinese Annals of Mathematics Series B 33(6) (2012), 877-888.
    
    \bibitem{lau2011sums}
    Y. K. Lau, and  G. Lü,  {\it Sums of Fourier coefficients of cusp forms}, The Quarterly Journal of Mathematics 62(3) (2011), 687-716.
    
    \bibitem{Newton-thorne-1}
    J. Newton, and J. A. Thorne,  {\it Symmetric power functoriality for holomorphic modular forms}, Publications mathématiques de l'IHÉS  134(1) (2021), 1-116.
    
    \bibitem{Newton-thorne-2}
    J. Newton, and J. A. Thorne,  {\it Symmetric power functoriality for holomorphic modular forms, II}, Publications mathématiques de l'IHÉS 134(1) (2021), 117–152.
    
    \bibitem{anubhav2022-1}
    A. Sharma and A. Sankaranarayanan, {\it Discrete mean square of the coefficients of symmetric square $L$-functions on certain sequence of positive numbers}, Research in Number Theory 8(1) (2022), 1-13.
    
    \bibitem{anubhav2022-2}
    A. Sharma and A. Sankaranarayanan,  {\it Higher moments of the Fourier coefficients of symmetric square L-functions on certain sequence}, Rendiconti del Circolo Matematico di Palermo Series 2 (2022), 1-18.
    
    \bibitem{anubhav2022-3}
    A. Sharma and A. Sankaranarayanan,  {\it Average behavior of the Fourier coefficients of symmetric square $L$-function over some sequence of integers}, Integers (Accepted).
    
    \bibitem{tang2013estimates}
    H. Tang,  {\it Estimates for the Fourier coefficients of symmetric square $L$-functions}, Archiv der Mathematik 100(2) (2013), 123-130.
    
    \bibitem{zhai2013average}
    S. Zhai,  {\it Average behavior of Fourier coefficients of cusp forms over sum of two squares}, Journal of Number Theory 133(11) (2013), 3862-3876.
    
    
    
    
   
    
   
    
    
    
    
    
    
    
    
    
   
    





















\end{thebibliography}
\end{document}